\documentclass[12pt,reqno]{amsart}
\usepackage{verbatim,amssymb,amsmath,amscd,latexsym,amsbsy,mathrsfs}
\usepackage{enumitem}   
\usepackage{amsthm}
\usepackage{mathtools}
\usepackage{hyperref}
\usepackage[all,cmtip]{xy}
\hypersetup{citecolor=blue}
\newtheorem{thm}{Theorem}[section] \newtheorem{pro}[thm]{Proposition}

\newtheorem{cor}[thm]{Corollary}
\numberwithin{equation}{section}

\newtheorem{proposition}[thm]{Proposition}
\newtheorem{lm}[thm]{Lemma}
\newtheorem{co}[thm]{Corollary}
\theoremstyle{remark} 
\theoremstyle{definition} 
\newtheorem{rmk}[thm]{Remark} 
\newtheorem*{note}{Note}
\newtheorem*{notation}{Notation}
\newtheorem{example}[thm]{Example}
\newtheorem{definition}[thm]{Definition}

\DeclareMathOperator*{\spec}{Spec} 
\DeclareMathOperator*{\card}{card} 
\DeclareMathOperator*{\Aut}{Aut} 
\DeclareMathOperator*{\Gal}{Gal} \DeclareMathOperator*{\Ind}{Ind}
 
\DeclareMathOperator*{\rank}{rank}\DeclareMathOperator*{\Inn}{Inn}
\DeclareMathOperator*{\Spec}{Spec}

\DeclareMathOperator*{\Pic}{Pic}
\DeclareMathOperator*{\Div}{Div}

\DeclareMathOperator*{\e}{\acute{e}t}
\DeclareMathOperator*{\id}{Id}
\DeclareMathOperator*{\dv}{div}

\newcommand{\chash}{\mathcal{\#}} 
 \newcommand{\QQ}{\mathbb{Q}}
\newcommand{\ZZ}{\mathbb{Z}} \newcommand{\Aff}{\mathbb{A}}
\newcommand{\PP}{\mathbb{P}} \newcommand{\FF}{\mathbb{F}}

\newcommand{\cO}{\mathcal{O}}

\newcommand{\onto}{\twoheadrightarrow}

\newcommand{\Z}{\mathbb{Z}}

\begin{document}
\title[Criterion for solving embedding problems for fundamental group]{A criterion for solving embedding problems for the \'etale fundamental group of curves}
\author{Manish Kumar} \author{Poulami Mandal}
\begin{abstract}
Let $C$ be an affine curve over an algebraically closed field $k$ of
characteristic $p>0$. Given an embedding problem $(\beta:\Gamma\longrightarrow G,
\alpha: \pi^{\e}_1(C)\longrightarrow G)$ for $\pi_1^{\e}(C)$ where $\beta$ is a surjective homomorphism of finite groups with prime-to-$p$ kernel $H$, we discuss when an $H$-cover of
the $G$-cover of $C$ corresponding to $\alpha$ is a solution. When $H$ is
abelian and $G$ is a $p$-group, some necessary and sufficient conditions
for the solvability of the embedding problems are given in terms of the action of $G$ on a certain generalization of $m$-torsion of the Picard group.
\end{abstract}
\maketitle

\section{Introduction}
Let $C$ be a smooth affine curve over an algebraically closed field $k$ of characteristic $p>0$. The structure of the \'etale fundamental group of $C$ is not quite well understood. The base point which will be omitted in the notation for \'etale the fundamental group is the generic geometric point of $C$. This profinite group $\pi_1^{\e}(C)$ is not topologically finitely generated. Hence even though we know (\cite{Ray}, \cite{Har}) all the finite quotients of $\pi_1^{\e}(C)$, they don't determine the group. There have been attempts to understand $\pi_1^{\e}(C)$ in various ways (\cite{Pop-halfRE} and \cite{BK}). For a survey on coverings and the \'etale fundamental group see \cite{HOPS}.

One way to address the problem is to understand which embedding problems for $\pi_1^{\e}(C)$ have solutions (see Section \ref{section1} for definition). An embedding problem $(\beta:\Gamma\longrightarrow G,
\alpha: \pi^{\e}_1(C)\longrightarrow G)$ for $\pi_1^{\e}(C)$ consists of surjective group homomorphisms $\alpha$ and $\beta$ with $\Gamma$ a finite group. Let $H$ be the kernel of $\beta$. It was shown by Pop (\cite[Theorem B]{PF}) and Harbater (\cite[Corollary 4.6]{Har2}) that when $H$ is a quasi-$p$ group (i.e. generated by its Sylow-$p$ subgroups), then the embedding problem has $\card(k)$ many solutions. But when $H$ is prime-to-$p$ the number of solutions is finite. Although it is not clear when a solution exists and how many distinct solutions there are for the embedding problem.  In \cite{HS} and \cite{K} solutions to prime-to-$p$ embedding problems restricted to certain open subgroups of $\pi_1^{\e}(C)$ were investigated. 
When $H\cong (\Z/m\Z)^r$ where $m$ is coprime to $p$ and $r\ge 1$, then $H$ is a $G$-representation over $\Z/m\Z$. We show that the solutions to the embedding problems with kernel $H$ are in bijection with those $G$-subrepresentations of $P_m(U)$ which are isomorphic to $H$ (see Theorem \ref{solutionsEP-submodule} and Corollary \ref{NSExt}). Here the $G$-Galois \'etale cover $U\longrightarrow C$ is determined by $\alpha: \pi^{\e}_1(C)\longrightarrow G$ and $P_m(U)$ is a generalization of the $m$-torsion of the Picard group associated to $U$ (see Definition \ref{GPic} and Section \ref{G-actionPm} for $G$-action on $P_m$). This result has some consequences for ``effective subgroups'' for an embedding problem (Corollary \ref{4.1}).

When the Galois group of the cover $U\longrightarrow C$ associated with the above embedding problem is a cyclic $p$-group, we obtain some more concrete results on when they have solutions and how many distinct solutions exist (see Theorem \ref{5.4}).
In \cite{G}, for a prime number $l$ other than $p$,  $(\ZZ/l\ZZ)^{\oplus n}\rtimes\ZZ/p\ZZ$-Galois covers of $\PP^1_k$ ramified only at infinity with minimal genus were constructed. Here, for a finite abelian group $H$, we describe some conditions (Corollaries \ref{finalcor}, \ref{5.6}) for existence of $H\rtimes\ZZ/p^a\ZZ$-Galois covers of smooth connected projective curves \'etale away from a finite set of points, dominating a given $\ZZ/p^a\ZZ$-Galois cover.

Recall that if the characteristic of $k$ is zero, then $\pi_1^{\e}(C)$ is the free profinite group over $2g_C+r_C-1$ elements, where $g_C$ is the genus of the smooth completion $X$ of $C$ and $r_C$ is the number of closed points in $X\setminus C$. The results described in Sections \ref{section1} and \ref{section3} hold even in characteristic zero case. But perhaps one could also obtain those results directly using group theory.

\section{Galois covers and pullback of coverings} \label{section1}
A morphism $\phi:Y\longrightarrow X$ between smooth curves is called a \emph{cover} if it is finite and generically smooth (separable). If $G$ is a finite group, then a $G$-Galois cover is a cover $Y\longrightarrow X$ together with a homomorphism $G\longrightarrow \Aut(Y/X )$ via which $G$ acts simply transitively on all the generic geometric fibres. The cover $\phi$ is called an \emph{\'etale cover} if it is also an \'etale morphism.
Unless otherwise stated, we will assume that the covers of curves are connected. Let $k$ be an algebraically closed field of characteristic $p>0$. Let $C$ be a smooth irreducible curve over $k$. A finite group $G$ is called quasi-$p$ if $G$ is generated by the $p$-sylow subgroups of $G$ and it is called a prime-to-$p$ group if $p$ is coprime to $|G|$. An \emph{embedding problem} (EP) $\mathcal{E}$ for the \'etale fundamental group $\pi_1^{\e}(C)$ is a pair of epimorphisms $\mathcal{E}=(\beta:\Gamma\longrightarrow G, \alpha:\pi_1^{\e}(C)\longrightarrow G )$ where $\Gamma$ and $G$ are finite groups.
\[
\xymatrix{
&&& \pi_1^{\e}(C)\ar@{.>>}[dl]_? \ar@{->>}[d]^\alpha\\
1\ar[r]& H\ar@{^{(}->}[r] & \Gamma \ar@{->>}[r]_\beta & G\ar[r] \ar[d] & 1\\
&&& 1
}\label{eq:E}\tag{$\mathcal{E}$}
\]
A \emph{proper solution} to $\mathcal{E}$ is an epimorphism from $\pi_1^{\e}(C)$ to $\Gamma$ so that the above diagram commutes.
Here $H=\ker(\beta)$. The EP $\mathcal{E}$ is said to be prime-to-$p$ (resp. quasi-$p$) if $H$ is a prime-to-$p$ (resp. quasi-$p$) group.  
A subset $B\subset H$ will be called a \emph{relative generating set} for $H$ in $\Gamma$ if for every subset $T\subset \Gamma$ such that $H \cup T$ generates $\Gamma$, the subset $B\cup T$ also generates $\Gamma$.
The\emph{ relative rank} of $H$ in $\Gamma$ is the smallest non-negative integer $\mu :=\rank_\Gamma(H)$ such that there is a relative generating set for $H$ in $\Gamma$ consisting of $\mu$ elements.

 Let $\mathcal{E}=(\beta, \alpha)$ be an EP. Let $\phi:X \longrightarrow C$ be the \'etale $G$-Galois cover corresponding to $\alpha$. A solution to $\mathcal{E}$ is an \'etale $H$-Galois cover $Y\longrightarrow X$ such that the composition $Y\longrightarrow C$ is a $\Gamma$-Galois cover. Let $NS(\mathcal{E})$ denote the number of equivalence classes of proper solutions to $\mathcal{E}$. Here we consider two solutions $\gamma_1$ and $\gamma_2$ of $\mathcal{E}$ to be equivalent if $\ker(\gamma_1)=\ker(\gamma_2)$. In other words, $NS(\mathcal E)$ counts the number of distinct $H$-covers of $X$ which become $\Gamma$ covers of $C$.
Note that if $\mathcal{E}$ is a nontrivial quasi-$p$ EP then $NS(\mathcal{E})$ is infinite (\cite[Corollary 4.6]{Har2}). If $\mathcal{E}$ is a prime-to-$p$ EP then $NS(\mathcal{E})$ is a finite number. This is because there are only finitely many $H$-Galois \'etale covers of $X$ when $H$ is a prime-to-$p$ group (\cite[XIII, Corollary 2.12, page 392]{SGA1}). Also note that if $G$ is trivial then $NS(\mathcal E)$ is simply the number of surjective group homomorphism $\pi_1^{\e}(C)\longrightarrow H$ divided by $|\Aut(H)|$.

Let $Z$ be a normal variety over $k$. Let $G$ be a finite group and $\pi:V\longrightarrow Z$ be a $G$-Galois cover of $Z$, \'etale over a non-empty open subset $U$ of $Z$. Let $H$ be a finite group and $\psi:W\longrightarrow V$ be a $H$-Galois cover of $V$, \'etale over $\pi^{-1}(U)$. Let $\sigma \in$ $\Aut(V/Z)=G$. Consider the pullback $W_\sigma$ of $W$:
\[
\xymatrix{
W_\sigma:=V\times_V W \ar[r]^-{\tilde{\sigma}} \ar[d]_{\psi_\sigma}
& W \ar[d]^\psi \\
V \ar[r]^\sigma & V
}
\]
Then $\psi_\sigma:W_\sigma\longrightarrow V$ is also an $H$-Galois cover, \'etale over $\pi^{-1}(U)$ and $\tilde{\sigma}$ is an isomorphism of schemes over $k$.

\begin{proposition}\label{fp}
The $H$-covers $W_\sigma\longrightarrow V$ are isomorphic to $W\longrightarrow V$, $ \forall \sigma \in$ $\Aut(V/Z)$ if and only if the composition $W\xrightarrow{~\psi~} V \xrightarrow{~\pi~} Z$
 is also Galois.
\end{proposition}

\begin{proof}
First we assume that for $\sigma\in$ $\Aut(V/Z)$, $\phi_\sigma:W \longrightarrow W_\sigma$ is an isomorphism of $H$-covers of $V$. We need to show that the field extension $k(W)/k(Z)$ induced by $k(Z)\hookrightarrow k(V)\hookrightarrow k(W)$ is Galois.
Let $\tau:k(W)\longrightarrow \overline{k(W)}$ be any field embedding in the algebraic closure $\overline{k(W)}$ that fixes $k(Z)$.
It is enough to prove that $\tau(k(W))=k(W)$. Since $k(V)/k(Z)$ is Galois $\tau(k(V))=k(V)\subset \tau(k(W))$. Let $\sigma=\tau|_{k(V)}$ and it defines an element of Aut$(V/Z)$. Let $\psi':W'\longrightarrow V$ be the normalization of $V$ in $\tau(k(W))$. Also, $\tau$ induces an isomorphism $W'\longrightarrow W$ and the following diagram is cartesian. 
\[
\xymatrix{
W' \ar[r]^-{\tau} \ar[d]_{\psi'}
& W \ar[d]^\psi \\
V \ar[r]^\sigma & V
}
\]
Hence $W'=W_\sigma$, $\psi'=\psi_\sigma$ and by hypothesis $\phi_{\sigma}:W\longrightarrow W'$ is an isomorphism of $H$-covers of $V$. Hence $k(W)=k(W')=\tau(k(W))$.

Conversely, suppose the composition $W\xrightarrow{\psi} V\xrightarrow{\pi} Z$ is Galois. 
Let $\sigma \in \text{Aut}(V/Z)$ be non-identity and $W_\sigma$ be the corresponding pullback of $W$. We want to define an isomorphism $\phi_{\sigma}:W\longrightarrow W_{\sigma}$ such that $\psi_{\sigma}\circ\phi_{\sigma}=\psi$. Let $k(W)$ be the splitting field of a polynomial $f\in k(V)[x]$. Then $k(W_{\sigma})$ is the splitting field of $\sigma(f)\in k(V)[x]$. Every root of $\sigma(f)$ is $k(Z)$-conjugate of a root of $f$ and since $k(W)/k(Z)$ is Galois, $\sigma(f)$
 splits in $k(W)$. Hence by comparing degrees $k(W)$ is also the splitting field of $\sigma(f)$. Hence there is an isomorphism $\phi_{\sigma}:k(W_{\sigma})\longrightarrow k(W)$ fixing $k(V)$. This induces the isomorphism $\phi_{\sigma}:W\longrightarrow W_{\sigma}$ of $V$-schemes.
\end{proof}

Let $\phi_\sigma:W\longrightarrow W_\sigma$ be an isomorphism of covers over $V$, for all $\sigma\in \Aut(V/Z)$.  Then by the proposition above, the composition $\pi\circ\psi:
\xymatrix@1{
W \ar[r]^\psi & V \ar[r]^\pi & Z 
}$
 is Galois. For $\sigma\in\Aut(V/Z)$, we get the following commuting diagram:
\[\xymatrix{
W\ar[r]^{\phi_\sigma} \ar[d]_\psi \ar@/^1.5pc/[rr]^{\tilde{\sigma}\circ\phi_\sigma} & W_\sigma \ar[r]^{\tilde{\sigma}} \ar[d]^{\psi_\sigma} & W\ar[d]^\psi\\
V\ar@{=}[r] \ar[d]_\pi & V\ar[r]^\sigma \ar[d]^\pi & V\ar[d]^\pi\\
Z\ar@{=}[r] & Z\ar@{=}[r] & Z
}\]
Let $\Phi_\sigma:=\tilde{\sigma}\circ\phi_\sigma :W\longrightarrow W$. Clearly, $\Phi_\sigma$ is an automorphism of $W$ and $\psi\circ\Phi_\sigma=\sigma\circ\psi$, $(\pi\circ\psi)\circ\Phi_\sigma =\pi\circ\psi$.
Hence $\Phi_\sigma \in \Aut(W/Z)$ and it is a lift of $\sigma$. When $\sigma=$ $\id_V\in \Aut(V/Z)$, we choose its lift to be $\id_W$.

\textbf{The $G$-action on $H$: }For all $\sigma \in G=\Aut(V/Z)$ we fix a lift $\Phi_\sigma$ as above. When $H$ is abelian, the right action of $G$ on $H$ is given by (see \cite[Section 6.6]{W}):
\begin{equation}
h\cdot\sigma := \Phi_\sigma^{-1}\circ h\circ\Phi_\sigma, \forall \sigma\in G \text{ and } h\in H \label{eq:ac}
\end{equation} 

\begin{rmk}
 When $H$ is abelian with the $G$-action as above, $H$ becomes a $G$-module. Then the equivalence classes of extensions of $G$ by $H$ are in one-to-one correspondence with the cohomology group $H^2(G; H)$ (\cite[Theorem 6.6.3]{W}). The factor set $[\cdot,\cdot]:G\times G\longrightarrow H$ given by $[\sigma,\tau]:= \Phi_{\sigma\tau}^{-1}\Phi_\sigma\Phi_{\tau}$ is a 2-cocycle and its image in $H^2(G; H)=Z^2(G, H)/B^2(G, H)$ corresponds to the Galois group $\Gamma$ of $W\longrightarrow Z$.
 For general $H$, the map $\alpha:G\longrightarrow \Aut(H)$ defined by $(\alpha(\sigma))(h)=\Phi_\sigma^{-1}\circ h\circ\Phi_\sigma$ induces the group homomorphism $\theta : G\longrightarrow \Aut(H)/\Inn(H)$ defined by $\theta(\sigma):=[\alpha(\sigma)]$. Then $(H,\theta)$ is a ``$G$-kernel with centre $Z(H)$''. This $\theta$ induces a group homomorphism $\theta_0:G\longrightarrow \Aut(Z(H))$. 
 Using the bijective correspondence between the classes of $(G,H)$-equivalent extensions of $G$  with $H^2(G,Z(H))$ (see \cite[Theorem 11.1]{ES} for more details)
 one can similarly find the cohomology class corresponding to $\Gamma$.
 \end{rmk}

Let $C$ be a smooth curve, $\alpha:\pi_1^{\e}(C)\longrightarrow G$ be an epimorphism and $H$ be a finite abelian group. Let $a:G\longrightarrow \Aut(H)$ be a fixed action of $G$ on $H$. 
Let $NSExt(a,\alpha)$ denote the sum of the number of solutions to embedding problems $(\beta,\alpha)$ where $\beta$ runs over all extensions of $G$ by $H$ given by $a$. Since $H^2(G,H)$ classifies all such extensions, we have the following formula. For $e\in H^2(G,H)$, let $\Gamma_e$ denote the extension and $\beta_e:\Gamma_e\longrightarrow G$ denote the epimorphism. Then
$$NSExt(a,\alpha)=\sum_{e\in H^2(G,H)}NS(\beta_e,\alpha).$$

\section{Cyclic covers prime-to-$p$}\label{section3}

Let $Y$ be a smooth projective curve of genus $g$ over an algebraically closed field $k$ of characteristic $p>0$. Let $m>1$ be coprime to $p$. Let $S$ be a finite set in $Y$ with $r$ elements. The following definitions and notations can be found in \cite[Chapter 3]{EV} and \cite[Section 3]{T}. 
We denote by $\Pic(Y)$ the Picard group of $Y$, by $\Div(Y)$ the Cartier divisors of $Y$, by $\Div_\mathbb{Q}(Y)=\Div(Y)\otimes_\mathbb{Z}\mathbb{Q}$,  by $\mathbb{Z}[S]$ the subgroup of divisors whose supports are contained in $S$, which can be identified with the free $\mathbb{Z}$-module with basis $S$.

Let $\Delta=\sum_{i=1}^nq_iD_i \in \Div(Y)\otimes_\mathbb{Z}\mathbb{Q}$, where $q_i\in\mathbb{Q}$, $D_i$ is a prime divisor and $n\in\mathbb{Z}_{\geq 0}$. Then let $[\Delta]$ denote $\sum_{i=1}^n[q_i]D_i\in \Div(Y)$, where $[q_i]=$ the integral part of $q_i$. For $L\in\Pic(Y)$, $m\ge 1$ and an effective Cartier divisor $D$ such that $L^{\otimes m}\cong\mathcal{O}(-D)$ where $\cO$ is the structure sheaf on $Y$. Define $L^{(i,D)}:=L^{\otimes i}\otimes\mathcal{O}([\frac{i}{m} D])$.
\begin{definition}\label{GPic}
We define
\begin{equation}
P_m(Y\setminus S)=\frac{\{([L],D)\in \Pic(Y)\oplus\mathbb{Z}[S]|L^{\otimes m}\cong\mathcal{O}(-D)\}}{\{([\mathcal{O}(-D),mD)|D\in\mathbb{Z}[S]\}} 
\end{equation}
as an $m$-torsion abelian group. 
\end{definition}
When $S=\emptyset$ then this is the $m$-torsion subgroup of the Picard group of $Y$. In general, it is an extension of $P_m(Y)$.
Note that $P_m=P_m(Y\setminus S)\cong H^1_\text{et}(Y\setminus S, \Z/m\Z)\cong(\mathbb{Z}/m)^{\oplus 2g+r-1+b^{(2)}}$ (See \cite[Section 3]{T}). Here $b^{(2)}=1$ if $r=0$, $b^{(2)}=0$ otherwise.
 
For simplicity, we also write elements of $P_m$ as $([L],D)$. Note that these elements represent $(\mathbb{Z}/m\ZZ)$-Galois \'etale (possibly disconnected) covers of $Y\setminus S$. We may assume that $D$ is an effective Cartier divisor with support in $S$ since the multiplicity at each point in $S$ can be chosen from $\{0,1,\ldots,m-1\}$. We fix an isomorphism $L^{\otimes m}\cong\mathcal{O}(-D)$.
The isomorphism allows one to define an $\cO$-algebra structure on $\oplus_{i=0}^{m-1} L^{(i,D)}$ and
 $\Spec_{\cO}(\oplus_{i=0}^{m-1} L^{(i,D)})$ is the $m$-cyclic (possibly disconnected) cover corresponding to (the equivalence class of) $([L],D)$ (See \cite[Chapter 3]{EV} for more details). 
\begin{definition} 
We say a subset $B$ of a finite group is of \emph{type T1} if subgroups generated by the elements in any two disjoint sets of $B$ intersect trivially. We say that the reduced and irreducible covers $X_1\longrightarrow Y, \ldots, X_n\longrightarrow Y$ are \emph{mutually linearly disjoint} if the fibre product $X_1\times_Y X_2\times\ldots\times_Y X_n$ is an integral scheme. 
\end{definition}
\begin{notation}
If $G_1$ is a subgroup of $G$, and if $V\longrightarrow Y$ is a $G_1$-Galois cover, then there is an induced $G$-Galois
cover $\Ind^G_{G_1}(V)\longrightarrow Y$, which consists of a disjoint union of $[G : G_1]$ copies of $V$, indexed by
the cosets of $G_1$ in $G$. More precisely, $\Ind^G_{G_1}(V)=(G\times V)/\sim$ where $(g,v)\sim (gh,h^{-1}v)$ for $g\in G$, $h\in G_1$ and $v\in V$.
\end{notation}
The part (i) of the proposition below is quite standard but we include a proof for the convenience of the readers.
\begin{proposition}\label{C}
For $Y$ and $P_m$ as above:
\begin{enumerate}[label=(\roman*)]
 \item An element of order $m$ in $P_m$ corresponds to a connected $m$-cyclic cover.
 \item Let $B\subset P_m$ be such that every element of $B$ is of order $m$ and for $\lambda\in B$ let $V_{\lambda}\longrightarrow Y$ be the $m$-cyclic cover corresponding to $\lambda$. Then $B$ is of type T1 iff the set of covers $\{V_{\lambda}\longrightarrow Y:\lambda\in B\}$ are mutually linearly disjoint.  
 \item  Let $B$ be as above of type T1. Let $\zeta$ be an element in the subgroup generated by $B$. Each connected component of the cover $V_{\zeta}\longrightarrow Y$ is dominated by the normalization of the fibre product of the covers $V_{\lambda}\longrightarrow Y$ for $\lambda\in B$.
 \item Let $B$ be as above of type T1 and $\mu$ be a primitive $m$-th root of unity in $k$. The subgroup generated by $B$ acts on the normalization of the cover $\times_{\lambda\in B}V_\lambda\longrightarrow Y$ naturally which extends the automorphism defined by $\lambda=([L],D)\in B$ of the cover $\Spec_{\cO}(\oplus_{i=0}^{m-1} L^{(i,D)})=V_{\lambda}\longrightarrow Y$ given by the multiplication of $\mu^i$ on sections of $L^{(i,D)}$.
\end{enumerate}
\end{proposition}
\begin{proof}

Let $\psi:W\longrightarrow Y$ be a smooth disconnected cover of $Y$, \'etale over $Y\setminus S$ with an $m$-cyclic group action. Then all connected components of $W$ are isomorphic and a connected component $W_1$ of $W$ is an $n$-cyclic cover of $Y$ and $nl=m$ where $l$ is the number of connected components of $W$. Then $\exists([N],F)\in P_n$ corresponding to $W_1\longrightarrow Y$. Then $W$ corresponds to $([N],lF) \in P_m$ and the order of $([N],lF)$ in $P_m$ is $n$.

Conversely, any element of order $n<m$ in $P_m$ is of the form $([L^{\otimes a}],aD)$, for some $([L],D)$ of order $m$ in $P_m$ and $n=\frac{m}{\gcd(a,m)}$. The cover $W$ corresponding to $([L^{\otimes a}],aD)$ is a disjoint union of $\frac{m}{n}$ number of isomorphic $n$-cyclic covers, each given by the normalization of $\Spec_{\cO}(\oplus_{i=0}^{n-1} L^{\otimes ai})$.

For (ii) and (iii) we use induction on $|B|$. Suppose $([L],D)$ and $([M],E)$ are in $B$ and for some $a$ and $b$ let  $([L^{\otimes a}],aD)=([M^{\otimes b}],bE)$ be order $n>1$. Let $W, W',W''$ be the $m$-cyclic covers corresponding to $([L],D), ([M],E), ([L^{\otimes a}],aD)$, respectively. Let $N=L^{\otimes a}$. Then a connected component $V$ of $W''$ corresponds to $([N],a'D)\in P_n$ for some $a'$. Moreover $W$ and $W'$ dominate $V$ (the $\cO$-algebra defining $V$ is a subalgebra of $\cO$-algebra defining $W$ and $W'$). 

Conversely suppose $W_1$ and $W_2$ are covers of $Y$ corresponding to $([L],D)$ and $([M],E)$. If they are not linearly disjoint then they both dominate a connected cover $V\longrightarrow Y$ of degree $n>1$. Let $([N],F)\in P_n$ be the element corresponding to $V$. Then $L^{\otimes a}=N=M^{\otimes b}$ and the cover associated to $([L^{\otimes a}],aD)$ is $\Ind^{\mathbb{Z}/m\Z}_{\frac{m}{n}\mathbb{Z}/m\Z}(V)=\Ind^{\mathbb{Z}/m\Z}_{\mathbb{Z}/n\Z}(V)$. Hence $\{([L],D),([M],E)\}$ is not of type T1.

For (iii), let $W_1$ and $W_2$ be covers of $Y$ corresponding to $([L],D)$ and $([M],E)$ and $W$ be the normalization of $W_1\times_Y W_2$. Then $W$ is the normalization of $\spec(\oplus_{i=0}^{m-1}\oplus_{j=0}^{m-1}L^{(i,D)}\otimes M^{(j,E)})$. One verifies that the $\cO$-algebra associated to the connected component $W_3$ of the cover associated to $([L^{\otimes a}\otimes M^{\otimes b}],aD+bE)$ is a subalgebra of $\oplus_{i=0}^{m-1}\oplus_{j=0}^{m-1}L^{(i,D)}\otimes M^{(j,E)}$. This induces a dominating map from $W\longrightarrow W_3$.

When  $|B|=n+1>2$ the same argument works. Consider the covers $W_1\longrightarrow Y,\ldots ,W_n\longrightarrow Y$ for any $n$ elements $\lambda_1,\ldots, \lambda_n$ of $B$. Then by induction hypothesis, these are mutually linearly disjoint and their normalized fibre product $W$ is a connected cover of $Y$. Let $W_0\longrightarrow Y$ be the cover corresponding to the remaining element $\lambda_0$. By looking at the algebra associated with these covers, if $W\longrightarrow Y$ and $W_0\longrightarrow Y$ are not linearly disjoint then the common cover will correspond to a connected component of the cover associated to $a\lambda_0=a_1\lambda_1+\ldots a_n\lambda_n$ for some $0<a<m$ and some $0\le a_i\le m-1$. Again for (iii), like in the $n=2$ case note that the $\cO$-algebra defining the normalization of $W_0\times_Y W$ contains the $\cO$-algebra defining the cover corresponding to an element of the subgroup generated by $B$.

To prove (iv), we fix $\mu$, a primitive $m$-th root of unity in $k$. For $\lambda=([L],D)\in B$, let $\Spec_{\cO}(\oplus_{i=0}^{m-1} L^{(i,D)})=V_{\lambda}\longrightarrow Y$ be the smooth irreducible cover. The $\cO$-algebra homomorphism defined by $h(l)=\mu^{i}l$ for any local section $l$ in $L^{(i,D)}$ defines an element $h\in \Aut(V_{\lambda}/Y)$ (see \cite[Section 3.9]{EV}). Fix an isomorphism from $<\lambda>$ to $<h>$ by mapping $\lambda$ to $h$ to identify $<\lambda>$ with $\Aut(V_{\lambda}/Y)$. Now if $B=\{\lambda_1,\ldots,\lambda_n\}$, the $m$-cyclic covers $V_{\lambda_1},\ldots,V_{\lambda_n}$ of $Y$ are linearly disjoint. The abelian group $<B>=<\lambda_1>\times\ldots\times<\lambda_n>$ acts component-wise on the normalised fibre product of these covers.
\end{proof}

\subsection{$G$-action on $P_m$}\label{G-actionPm}
Let $X$ be a smooth connected projective curve and $S_X$ be a finite set of closed points of $X$. Let $\psi:V\longrightarrow X$ be a connected $G$-Galois cover for a finite group $G$ and $S_V=\psi^{-1}(S_X)$. Let $r_X=|S_X|$, $r_V=|S_V|$, $g_X$ and $g_V$, the genus of $X$ and $V$ respectively.
For $\sigma\in G$, $D=\sum_{v\in S_V} a_v v\in \mathbb{Z}[S_V]$, $\sigma^*D=\sum_{v\in S_V} a_v \sigma^{-1}v\in \mathbb{Z}[S_V]$ as $\sigma(S_V)=S_V$.
So $G$ acts on $P_m$ on the right by $([L],D)\mapsto ([\sigma^*L], \sigma^*D)$.  It is easy to see that this action preserves the group operation of $P_m$. 
Hence $P_m$ is a right $G$-module. Note that $P_m$ is naturally a $\Z/m\Z$-module with compatible $G$-action.

\begin{rmk}\label{P_mGalAction}
We choose a generating set $A=\{a_1,\ldots,a_N\}$ of type T1 of $P_m(V\setminus S_V)$ such that each $a_i$ is of order $m$. We also fix $\mu$, a primitive $m$-th root of unity in $k$.  Using Proposition \ref{C}(iv), we fix an isomorphism from $P_m$ to the Galois group of the normalization of the cover $\times_{i=1}^N W_i\longrightarrow V$, where $W_i\longrightarrow Y$ is the cover corresponding to $a_i$.
\end{rmk}

\begin{thm}\label{solutionsEP-submodule}
Let $H$ be a free $\Z/m\Z$-submodule of $P_m=P_m(V\setminus S_V)$. Then the following embedding problem
\[
\xymatrix{
&&& \pi_1^{\e}(X\setminus S_X)\ar@{.>>}[dl]_? \ar@{->>}[d]^\alpha\\
1\ar[r]& H\ar@{^{(}->}[r] & \Gamma \ar@{->>}[r]_\beta & G\ar[r] \ar[d] & 1\\
&&& 1
}
\]
has a solution for some $\Gamma$ and $\beta$ if $H$ is a $G$-submodule of $P_m$. Here $\alpha$ corresponds to the $G$-Galois cover $V\longrightarrow X$ \'etale away from $S_V$.
Conversely, given a solution $\gamma$ to the above embedding problem with $H$ a free $\Z/m\Z$-module, then the $H$-Galois \'etale cover of $V\setminus S_V$ induced from $\gamma$ comes from a $G$-stable subgroup of $P_m$ isomorphic to $H$ as $G$-modules.

Moreover, two solutions to the embedding problem lead to the same $G$-submodule of $P_m(V\setminus S_V)$ iff they are equivalent. 
\end{thm}

\begin{proof}
  Let $B\subset P_m=P_m(V\setminus S_V)$ be of type T1 consisting of elements of order $m$ and $H=<B>$.
  Suppose $H$ is a $G$-submodule of $P_m$. For $\lambda\in B$, let $W_{\lambda}\longrightarrow V$ be the $m$-cyclic cover corresponding to $\lambda$. Let $W\longrightarrow V$ be the normalized fibre product of these covers. Then by Proposition \ref{C}, $W\longrightarrow V$ is an $H$-Galois connected cover \'etale over $V\setminus S_V$. Since $H$ is a $G$-module for any $\sigma\in G$ and $\lambda\in B$, $\sigma^*\lambda\in H$. Hence $W\longrightarrow V$ dominates the covers $W_{{\sigma}^*\lambda}\longrightarrow V$ corresponding to $\sigma^*\lambda$ for every $\lambda\in B$. Moreover, $\sigma(B)$ consists of elements of order $m$ and is of type T1. Hence by Proposition \ref{C} the normalized fibre product $W^*_{\sigma}$ of the covers $W_{{\sigma}^*\lambda}\longrightarrow V$ is dominated by $W\longrightarrow V$. Comparing degrees we obtain that the covers $W^*_{\sigma}\longrightarrow V$ and $W\longrightarrow V$ are isomorphic. Since this is true for all $\sigma\in G$, by Proposition \ref{fp} $W\longrightarrow X$ is a Galois cover. Let $\Gamma$ be the Galois group. Then $\Gamma$ is an extension of $G$ by $H$.
  
  Since $W\longrightarrow X$ is Galois, for any $\sigma\in G$, we choose a lift $\Phi_\sigma=\tilde{\sigma}\circ\phi_\sigma$ from the Galois group $\Gamma$. Let $h\in\Gal(W/V)$ be the image of $\lambda\in<B>$ (by Proposition \ref{C}) and $h^\sigma$ be the pullback of $h$. Note that $\Phi_\sigma^{-1}$ induces an isomorphism from $W_\lambda$ to $W_{\sigma^*\lambda}$, the $m$-cyclic cover of $V$ corresponding to $\sigma^*\lambda$.
\[
\xymatrix{
W\ar[r]^{\phi_\sigma} \ar[d]^{\sigma^*h} & W_\sigma\ar[r]^{\tilde{\sigma}}\ar[d]^{h^\sigma} & W\ar[d]^h\\
W\ar[r]^{\phi_\sigma} & W_\sigma\ar[r]^{\tilde{\sigma}} & W
}
\]
 From the commuting diagram above, $h$ induces $\sigma^*h$, the image of $\sigma^*\lambda$ in $\Gal(W_{\sigma^*\lambda}/V)$.

 For a morphism $f$ of varieties over $k$, we let $f^{\chash}$ denote the map of rational functions.
 Let $B=\{\lambda_1,\ldots,\lambda_n\}$, $\lambda_i=([L_i],D_i)$, $W_i$ be the corresponding cover and $h_i$ be the image of $\lambda_i$ in $\Gal(W_i/V)$. Then $L_i^m\otimes \mathcal O_V(D_i)=\dv(s_i)$ for some $s_i$ in $k(V)$ and $k(W_i)$ is $k(V)(t_i)$ for some $t_i\in k(W)$ such that $t_i^m=s_i$ and  $h_i^\chash(t_j)=\mu^{\delta_{ij}}t_j$ where $\delta_{ij}=0$ if $i\neq j$, $\delta_{ij}=1$ if $i=j$.
 Note that $k(W_{\sigma^*\lambda_j})=k(V)(\Phi_\sigma^\chash(t_j))$ such that $(\Phi_\sigma^\chash(t_j))^m=\sigma^\chash(s_j)$. Clearly, $k(W)=k(V)(\Phi_\sigma^\chash(t_1),\ldots,\Phi_\sigma^\chash(t_n))$ and $(\sigma^*h_i)^\chash(\Phi_\sigma^\chash(t_j))=\mu^{\delta_{ij}}(\Phi_\sigma^\chash(t_j))$.
 
 Note that
 $$(\Phi_\sigma^{-1} h_i\Phi_\sigma)^\chash(\Phi_\sigma^\chash(t_j))=\Phi_\sigma^\chash h_i^\chash(\Phi_\sigma^\chash)^{-1}(\Phi_\sigma^\chash(t_j))=\mu^{\delta_{ij}}(\Phi_\sigma^\chash(t_j))=(\sigma^*h_i)^\chash(\Phi_\sigma^\chash(t_j)).$$ We have $\Phi_\sigma^{-1}\circ h_i\circ\Phi_\sigma=\sigma^*h_i$.
 So we obtain $h\cdot\sigma=\Phi_\sigma^{-1}\circ h\circ\Phi_\sigma=\sigma^*h$, for any $h\in \Gal(W/V)$. Hence, the group action of $G$ on $H\cong\Gal(W/V)$ is the same as the action of $G$ on $H$ as a $G$-submodule of $P_m$  and $W\longrightarrow X$ provides a solution to the embedding problem.

 For the converse, we choose a basis $B=\{h_1,\ldots,h_r\}$ of the free $\Z/m\Z$-module $H$. The solution $\gamma$ to the EP corresponds to the $\Gamma$-cover $W\to X$ which leads to the $H$-cover $W\longrightarrow V$. Let $W_i\longrightarrow V$ be the $\Z/m\Z$-cover $W/H_i$ where $H_i=<h_j:1\le j \le r, j\ne i>$.
 Note that $W$ is the normalization of fibre product of $W_i$'s. Since $W_i\longrightarrow V$ is a $\ZZ/m\ZZ$-cover, it corresponds to some $([L_i],D_i)$ in $P_m$. So the subgroup of $P_m$ generated by $\{([L_i],D_i) : 1\leq i\leq r\}$ is isomorphic to $H$. Note that this subgroup does not depend on the choice of the basis (by Proposition \ref{C}(iii)). By Proposition \ref{fp}, this subgroup of $P_m$ is $G$-stable under the $G$-action on $P_m$. Moreover, exactly like in the previous paragraph, the $G$-action on this subgroup of $P_m$ is the same as the $G$-action on $H$ induced from the EP. Hence $H$ is isomorphic to this subgroup $<\{([L_i],D_i) : 1\leq i\leq r\}>$ of $P_m$ as $G$-modules.

  Furthermore, for the last part of the theorem, observe that two solutions for the embedding problem are equivalent iff they induce the same $H$-cover $W\longrightarrow V$. Finally, the $H$-cover $W\longrightarrow V$ determines the subgroup of $P_m$ isomorphic to $H$ by the above.
\end{proof}

\begin{cor}\label{NSExt}
 Let $\alpha$ be as above, $H=(\Z/m\Z)^r$ be a $G$-module and $a:G\longrightarrow \Aut(H)$ be the associated action. Then $NSExt(a,\alpha)$ is the number of distinct $G$-submodules of $P_m(V\setminus S_V)$ isomorphic to $H$. 
\end{cor}

\begin{proof}
By Theorem \ref{solutionsEP-submodule}, there is a bijection between $G$-submodules $H$ of $P_m=P_m(V\setminus S_V)$ and equivalence classes of solutions to the embedding problems $(\beta,\alpha)$ where $\beta:\Gamma\longrightarrow G$ is an epimorphism and $\Gamma$ is any extension of $G$ by $H$.  
\end{proof}

An immediate consequence of Corollary \ref{NSExt} is the following corollary.
\begin{cor}\label{B}
There exists a unique extension $\Gamma_0$ of $G$ by $P_m$ such that the following embedding problem has a unique equivalence class of solutions.
\[
\xymatrix{
&&& \pi_1^{\e}(X\setminus S_X)\ar@{.>>}[dl]_? \ar@{->>}[d]^\alpha\\
1\ar[r]& P_m\ar@{^{(}->}[r] & \Gamma_0 \ar@{->>}[r] & G\ar[r] \ar[d] & 1\\
&&& 1
}
\]
\end{cor}

\begin{proof}
 Taking $H=P_m$ in the above corollary, we obtain a unique equivalence class of solutions to the embedding problem. 
\end{proof}

We are interested in finding $G$-submodules $H$ of $P_m$ as these give rise to the solutions to embedding problems for $\alpha:\pi_1^{\e}(X\setminus S_X)\longrightarrow G$ with some extension $\Gamma$ of $G$ by $H$.

Note that $S_V$ is a $G$-set. Let $S$ be a (possibly empty) $G$-stable subset of $S_V$.
\begin{pro}
 The group $P_m(V\setminus S)$ is a $G$-submodule of $P_m(V\setminus S_V)$. In particular, there is an extension $\Gamma$ of $G$ by $P_m(V\setminus S)$ such that the embedding problem $(\beta:\Gamma\longrightarrow G, \alpha)$ has a solution.
\end{pro}
\begin{proof}
 By definition of $P_m$ there is a natural inclusion of $H=P_m(V\setminus S)$ in $P_m(V\setminus S_V)$. Since $S$ is a $G$-set $H$ is stable under the action of $G$ on $P_m(V\setminus S_V)$ defined in \ref{G-actionPm}. Hence $H$ is a $G$-submodule of $P_m(V\setminus S_V)$. The rest follows from Theorem \ref{solutionsEP-submodule}.
\end{proof}

Suppose $G_1\unlhd G$ and let $f_1:X_1\longrightarrow X$ be the normalization of $X$ in $k(V)^{G_1}$. Then the cover $V\longrightarrow X$ factors through $X_1\longrightarrow X$, $\Aut(X_1/X)=G/G_1$ and $\Aut(V/X_1)=G_1$.
\begin{pro}
Let $S_1$ be a subset of $f_1^{-1}(S_X)$ stable under the action of $G/G_1$. There is a natural $G$-equivariant homomorphism $P_m(X_1\setminus S_1)\longrightarrow P_m(V\setminus S_V)$. Moreover if $(|G_1^{ab}|,m)=1$ then this homomorphism is injective.  In particular there is an extension $\Gamma$ of $G$ by $P_m(X_1\setminus S_1)$ such that the embedding problem $(\beta:\Gamma\longrightarrow G, \alpha)$ has a solution.
\end{pro}
\begin{proof}
 By definition of $P_m$ there is a natural inclusion of $P_m(X_1\setminus S_1)$ in $P_m(X_1\setminus f_1^{-1}(S_X))$. Since $S_1$ is a $G/G_1$-set $P_m(X_1\setminus S_1)$ is stable under the action of $G/G_1$ on $P_m(X_1\setminus f_1^{-1}(S_X))$ defined in \ref{G-actionPm}. Hence $P_m(X_1\setminus S_1)$ is a $G$-submodule of $P_m(X_1\setminus f_1^{-1}(S_X))$. Let $h:V\longrightarrow X_1$ be the cover which composed with $X_1\longrightarrow X$ is $V\longrightarrow X$. Then for $([\mathcal{L}],{D})\in P_m(X_1\setminus f_1^{-1}(S_X))$, $([h^*\mathcal{L}], h^*D)$ is in $P_m(V\setminus S_V)$. This defines a $G$-equivariant map $h^*:P_m(X_1\setminus f_1^{-1}(S_X))\longrightarrow P_m(V\setminus S_V)$. Since the cover of $W\longrightarrow X_1$ defined by elements of $P_m(X_1\setminus f_1^{-1}(S_X))$ are Galois with abelian Galois group of order some power of $m$ and $G_1^{ab}$ is of order prime to $m$, $W\longrightarrow X_1$ and $V\longrightarrow X_1$ are linearly disjoint. Hence $h^*$ is injective.
 The remaining statement follows from Theorem \ref{solutionsEP-submodule}.
\end{proof}

\section{Effective subgroups}

Let $X^o$ be a smooth connected affine curve, $X$ be its smooth completion and $S_X=X\setminus X^o$. Let $\alpha:\pi_1^{\e}(X^o)\longrightarrow G$ be an epimorphism. Let $H$ be a finite group.
Let $Y\longrightarrow X$ be the $G$-Galois connected cover corresponding to $\alpha$. 
Given a finite index subgroup $\Pi\subset \pi_1^{\e}(X^o)$, we say an embedding problem $(\beta:\Gamma\onto G, \alpha)$ restricts to $\Pi$ if $\alpha(\Pi) = G$. If the restricted embedding problem has a proper solution then we say $\Pi$ is an \emph{effective subgroup} for $\mathcal{E}$. We use the notation and terminology of Serre (see \cite{Se}) for ramification filtration and upper jump.
The following is an immediate consequence of Theorem \ref{solutionsEP-submodule}.
\begin{cor}\label{4.1}
  Let $Z\longrightarrow X$ be a cover \'etale over $X^o$ such that the normalization $f:V\longrightarrow X$ of the fibre product $Y\times_X Z$ is connected. Also assume that there exists an $H$-cover $W^o$ of $V^o=f^{-1}(X^o)$ whose pullback by all $\sigma \in G$ is again the same $H$-cover $W^o\longrightarrow V^o$. Then there exist an extension $\beta:\Gamma\longrightarrow G$ of $G$ by $H$ such that $\pi_1^{\e}(Z^o)$ is an effective subgroup of $\pi_1^{\e}(X^o)$ for the EP $(\beta, \alpha)$.
\end{cor}

Let $H$ be a finite group and $a:G\longrightarrow \Aut(H)$ be an action of a finite group $G$ on $H$. Let $Y^o\longrightarrow \Aff^1$ be an \'etale $G$-Galois cover. Let $Z^o\longrightarrow \Aff^1$ be a $p$-cyclic \'etale cover and $g_Z$ be the genus of the smooth completion $Z$ of $Z^o$. Let $f:V\longrightarrow Z$ be the normalization of the fibre product $Y\times_{\PP^1} Z$ and $V^o=f^{-1}(Z^o)$.

\begin{thm} \label{4.2}
 Let $m$ be prime to $p$, $a:G\longrightarrow \Aut(H)$ be a representation over $\Z/m\Z$ and $r$ be the size of the smallest generating set of $H$ as a $G$-module. If $g_Z\ge r/2$ and all the upper jumps of $Y^o\longrightarrow \Aff^1$ at $\infty$ are different from $1+2g_Z/(p-1)$ then $a:G\longrightarrow \Aut(H)$ is a subrepresentation of $P_m(V^o)$. 
\end{thm}

\begin{proof}
 Note that the upper jump of $Z\longrightarrow \PP^1$ over $\infty$ is $1+2g_Z/(p-1)$ (see \cite[Remark 18]{K}). Also $f:V\longrightarrow Z$ is a connected $G$-Galois cover \'etale over $Z^o$. Let $\alpha:\pi_1^{\e}(Z^o)\longrightarrow G$ be corresponding epimorphism and $\beta:\Gamma \longrightarrow G$ be any extension of $G$ by $H$ with respect to the action $a$. Note that the relative rank of $H$ in $\Gamma$ is at most $r$. By \cite[Corollary 17]{K} the EP $(\beta,\alpha)$ has a proper solution. Hence by Theorem \ref{solutionsEP-submodule}, $H$ is a $G$-submodule of $P_m(V^o)$.
\end{proof}

\begin{example}
Let $Y^o\longrightarrow \Aff^1$ be a $G$-Galois \'etale cover and $Z^o\longrightarrow\mathbb{A}^1_k$ be a $\mathbb{Z}/p\mathbb{Z}$-Galois \'etale cover such that $g_Y>g_Z|G|$ ($g_Z$ is genus of $Z\supset Z^o$, $g_Y$ is genus of $Y\supset Y^o$) and $k(Z)\cap k(Y)=k(x)$. Let $g_V$ be genus of $V$, the normalization of $Y\times_{\PP^1} Z$ and $S$ be the branch locus of $V\longrightarrow Z$. Here $m$ is prime to $p$ and the $G$-action on $P_m=P_m(V^o)$ is the same as in \ref{G-actionPm}. Note that $g_V>g_Z|G|$ and hence $|P_m|>m^{2g_Z|G|}$.
Let $B\subset P_m$ be of cardinality at most $2g_Z$. Let $\Gamma$ be the semidirect product of $P_m$ and $G$. Then the $G$-submodule of $P_m$ generated by $B$ has cardinality at most $m^{2g_Z|G|}$ and hence is not the whole of $P_m$. This implies $B$ is not a relative generating set of $H$ in $\Gamma$. By Theorem \ref{solutionsEP-submodule} $\pi_1^{\e}(Z^o)$ is an effective subgroup for the EP $(\Gamma\longrightarrow G, \pi_1^{\e}(\Aff^1)\longrightarrow G)$. Note that if $H^2(G,P_m)=0$ then we are in the setup of \cite[Proposition 15]{K} and this example shows that the sufficient condition mentioned in \cite[Proposition 15]{K} is far from necessary.
\end{example}

\section{The case when $G$ is a cyclic $p$-group} \label{5}

Let $X$ be a smooth connected projective curve and $S_X$ be a finite set of closed points of $X$. Let $\psi:V\longrightarrow X$ be a connected $G$-Galois cover \'etale away from $S_X$ for a finite cyclic $p$-group group $G$, $S_V=\psi^{-1}(S_X)$ and $\alpha:\pi_1^{\e}(X\setminus S_X)\twoheadrightarrow G$ be the homomorphism corresponding to $\psi$. Let $l$ be a prime number other than $p$. For $b\ge 1$, $d_b$ will denote the order of $l$ in $(\mathbb{Z}/p^b\mathbb{Z})^*$.

\begin{lm}\label{Rep}
Every nontrivial irreducible $\mathbb{F}_l$-representation of $\mathbb{Z}/p^a\mathbb{Z}$ is of dimension $d_b$ for some $b\leq a$. 
\end{lm}

\begin{proof}
Let $G$ be the cyclic group of order $p^a$ and $\sigma$ be its generator. Clearly $\sigma^{p^a}=1_G$ implies that the minimal polynomial of $\sigma$ divides $x^{p^a}-1$. Now, $$x^{p^a}-1=(x-1)\prod_{b=1}^a Q_{p^b}(x)=(x-1)\prod_{b=1}^a\prod_{i=1}^\frac{p^{b-1}(p-1)}{d_b}P_{bi}(x),$$ where $Q_{p^b}(x)$ is the monic irreducible polynomial of a primitive $p^b$-th root of unity (in $\mathbb{C}$) over $\QQ$, $P_{bi}(x)$ are irreducible factors (over $\mathbb{F}_l$) of $Q_{p^b}(x)$ (\cite{LN}, Theorem 2.47) of degree $d_b$. Let $M$ be a nontrivial irreducible $G$-representation. Then $M$ is a simple $\FF_l[x]$-module where multiplication by $x$ is the nontrivial action by $\sigma$. Hence $M\cong \FF_l[x]/P_{bi}(x)$ for some $i\in\{1,\ldots,\frac{p^{b-1}(p-1)}{d_b}\}, b\in\{1,\ldots,a\}$, by structure theorem for modules over PID. Since $M$ is a nontrivial representation
$\FF_l[x]/(x-1)$ is ruled out. Hence the dimension of $M$ is $d_b=\deg(P_{bi})$.
\end{proof}

\begin{note}
Every nontrivial irreducible $\mathbb{F}_l$-representation of $\mathbb{Z}/p^a\mathbb{Z}$ is isomorphic to $\FF_l[x]/P_{bi}(x)$ for some $1\leq b\leq a $ and $1\leq i\leq p^{b-1}(p-1)/d_b$.
\end{note}

\begin{thm}
Let $G$ be a cyclic group of order $p^a$, $\sigma$ be a generator of $G$, $l$ be a prime number other than $p$, $H=(\mathbb{Z}/l\Z)^{n}$, $n_0$ (resp. $n_b$, $1\le b\le a$) be the dimension of $P_l(V\setminus S_V)^G$ (resp. $ \ker(Q_{p^b}(\sigma))\subset P_l(V\setminus S_V)$) over $\FF_l$. Then $n$ can be expressed as $n=u+\Sigma_{b=1}^av_bd_b$ for non-negative integers $u\le n_0$, $v_b\le n_b/d_b,\forall b\le a$ if and only if the embedding problem  $(\beta:H\rtimes_\theta G\twoheadrightarrow G, \alpha:\pi_1^{\e}(X\setminus S_X)\twoheadrightarrow G)$ has a proper solution for some group homomorphism $\theta:G\longrightarrow \Aut(H)$.
\end{thm}

\begin{proof}
By Lemma \ref{Rep}, $P_l(V\setminus S_V)\cong (\FF_l)^{n_0}\oplus(\oplus_{b=1}^a \oplus_{i=1}^{p^{b-1}(p-1)/d_b}(\FF_l[x]/P_{bi}(x))^{\gamma_{bi}})$ as $\FF_l[G]$-modules, where the non-negative integers $\gamma_{bi}$ satisfy $\Sigma_{i=1}^{p^{b-1}(p-1)/d_b} \gamma_{bi}d_b=n_b$.\\ 
Since $\gcd(|H|,|G|)=1$, any extension of $G$ by $H$ with respect to some $G$-action $\theta$ must be equivalent to $H\rtimes_\theta G$. Note that any basis of $H$ over $\FF_l$ is of type T1 consisting of elements of order $l$.

If $n$ can be expressed as above, then $H$ is identified with a $G$-stable subspace $(\FF_l)^u\oplus(\oplus_{b=1}^a \oplus_{i=1}^{p^{b-1}(p-1)/d_b}(\FF_l[x]/P_{bi}(x))^{\gamma'_{bi}})$ of $P_l(V\setminus S_V)$ where $\Sigma_{i=1}^{p^{b-1}(p-1)/d_b} \gamma'_{bi}=v_b$ and the embedding problem $(\beta:H\rtimes_\theta G\twoheadrightarrow G, \alpha:\pi_1^{\e}(X\setminus S_X)\twoheadrightarrow G)$ has a proper solution by Theorem \ref{solutionsEP-submodule} where $\theta$ depends on the choice of identification of $H$.

Conversely, if an embedding problem $(\beta:H\rtimes_\theta G\twoheadrightarrow G, \alpha:\pi_1^{\e}(X\setminus S_X)\twoheadrightarrow G)$ has a proper solution, then $H=\ker(\beta)$ must be a $\FF_l[G]$-submodule of $P_l(V\setminus S_V)$ (Theorem \ref{solutionsEP-submodule}). So  $H\cong (\FF_l)^u\oplus(\oplus_{b=1}^a \oplus_{i=1}^{p^{b-1}(p-1)/d_b}(\FF_l[x]/P_{bi}(x))^{\gamma'_{bi}})$  where $0\le u\le n_0 $, $0\le \gamma'_{bi}\le\gamma_{bi}$. Put $v_b=\Sigma_{i=1}^{p^{b-1}(p-1)/d_b} \gamma'_{bi}$. Then $v_b\le n_b/d_b$ and $n=u+\Sigma_{b=1}^a\Sigma_{i=1}^{p^{b-1}(p-1)/d_b}\gamma'_{bi}d_b =u+\Sigma_{b=1}^av_bd_b$.   
\end{proof}

We can replace $l$ by any square-free integer in the above corollary.

\begin{co} \label{finalcor}
Let $G$ be a cyclic group of order $p^a$, $\sigma$ be a generator of $G$, $m$ be a square free integer prime to $p$, $m=l_1\ldots l_T$, for distinct prime numbers $l_\tau$; $d_{b\tau}$ be the order of $l_\tau$ in $(\ZZ/p^b\ZZ)^*$, $H=(\mathbb{Z}/m\Z)^{n}$, $n_{0\tau}$ (resp. $n_{b\tau}$, $1\le b\le a$) be the dimension of $P_{l_\tau}(V\setminus S_V)^G$ (resp. $ \ker_\tau(Q_{p^b}(\sigma))\subset P_{l_\tau}(V\setminus S_V)$) over $\FF_{l_\tau}$. Then for each $\tau\le T$, $n$ can be expressed as $n=u_\tau+\Sigma_{b=1}^av_{b\tau}d_{b\tau}$ for non-negative integers $u_\tau\le n_{0\tau}$, $v_{b\tau}\le n_{b\tau}/d_{b\tau},\forall b\le a$ if and only if the embedding problem  $(\beta:H\rtimes_\theta G\twoheadrightarrow G, \alpha:\pi_1^{\e}(X\setminus S_X)\twoheadrightarrow G)$ has a proper solution for some group homomorphism $\theta:G\longrightarrow \Aut(H)$.
\end{co}

\begin{proof}
The abelian groups $P_m(V\setminus S_V), H$ have  unique $G$-stable decompositions: $P_m(V\setminus S_V)=\oplus_{\tau=1}^T P_{l_\tau}(V\setminus S_V)$, $H\cong \oplus_{\tau=1}^T (\mathbb{Z}/l_\tau\mathbb{Z})^{n}$. Now apply the above result for each $P_{l_\tau}(V\setminus S_V),H_\tau:=(\mathbb{Z}/l_\tau\mathbb{Z})^{n}$.
\end{proof}

We can count the number of equivalence classes of solutions for such embedding problems as well.

\begin{thm}\label{5.4}
Let $G$ be a cyclic group of order $p^a$, $l$ be a prime number different from $p$. Let $H=(\ZZ/l\ZZ)^{\oplus n}$ be a $G$-module, $\theta:G\longrightarrow \Aut(H)$ be the $G$-action, $\alpha:\pi_1^{\e}(X\setminus S_X)\twoheadrightarrow G$ be an epimorphism. Let $\gamma_{bi}$ (resp. $\gamma'_{bi}$) be the multiplicity of $\FF_l[x]/P_{bi}(x)$ in the $G$-module $P_l(V\setminus S_V)$ (resp. $H$).
Let $n_0$ (resp. $u$) be the dimension of $P_l(V\setminus S_V)^G$ (resp. $H^G$). Then $$ NSExt(\theta,\alpha)=[\Pi_{b=1}^a\Pi_{i=1}^{p^{b-1}(p-1)/d_b} (n_{bi}/n'_{bi})]\bar{n}/\overline{n'},$$ where $n_{bi} =\Pi_{r=o}^{\gamma'_{bi}-1}(\sum_{s=r}^{\gamma_{bi}-1}l^{d_bs})$, $n'_{bi} =\Pi_{r=o}^{\gamma'_{bi}-1}(\sum_{s=r}^{\gamma'_{bi}-1}l^{d_bs})$, $\bar{n}=\Pi_{r=o}^{u-1}(\sum_{s=r}^{n_0-1}l^s)$ and $\overline{n'}=\Pi_{r=o}^{u-1}(\sum_{s=r}^{u-1}l^s)$.
\end{thm}
\begin{proof}
Let $P_l(V\setminus S_V)\cong (\FF_l)^{n_0}\oplus(\oplus_{b=1}^a \oplus_{i=1}^{p^{b-1}(p-1)/d_b}(\FF_l[x]/P_{bi}(x))^{\gamma_{bi}})$ and fix an isomorphism $H\cong (\FF_l)^u\oplus(\oplus_{b=1}^a \oplus_{i=1}^{p^{b-1}(p-1)/d_b}(\FF_l[x]/P_{bi}(x))^{\gamma'_{bi}})$ (by Lemma 5.1). For $v\in \ker(P_{bi}(\sigma))\setminus\{0\}\subset P_l(V\setminus S_V) $, $\{v, \sigma v,\ldots, \sigma^{d_b-1}v\}$ is linearly independent and $\{\sum_{j=0}^{d_b-1}\alpha_j\sigma^jv| \alpha_j\in\FF_l\}$ is a $G$-stable irreducible subspace of $\ker P_{bi}(\sigma)$.
The number of $d_b$-dimensional $G$-stable subspace of $\ker(P_{bi}(\sigma))$ is $a_{bi}=\frac{l^{d_b\gamma_{bi}}-1}{l^{d_b}-1}=(l^{d_b(\gamma_{bi}-1)}+l^{d_b(\gamma_{bi}-2)}+\ldots+ l^{d_b}+1)$ and the same for $\ker_H(P_{bi}(\sigma))$ is $a'_{bi}=\frac{l^{d_b\gamma'_{bi}}-1}{l^{d_b}-1}$.
The number of ways to choose first $\gamma'_{bi}$ components up to $G$-module automorphism
$$
=\frac{a_{bi}(a_{bi}-1)(a_{bi}-\frac{l^{2d_b}-1}{l^{d_b}-1})(a_{bi}-\frac{l^{3d_b}-1}{l^{d_b}-1})\ldots  (a_{bi}-\frac{l^{(\gamma'_{bi}-1)d_b}-1}{l^{d_b}-1})}{a'_{bi}(a'_{bi}-1)(a'_{bi}-\frac{l^{2d_b}-1}{l^{d_b}-1})(a'_{bi}-\frac{l^{3d_b}-1}{l^{d_b}-1})\ldots  (a'_{bi}-\frac{l^{(\gamma'_{bi}-1)d_b}-1}{l^{d_b}-1})}
=\frac{n_{bi}}{n'_{bi}}.
$$
Similarly, the number of ways to choose first $u$ components of $P_l(V\setminus S_V)^G$ up to $G$-module isomorphism is $\bar{n}/\overline{n'}$. Then we get $[\Pi_{b=1}^a\Pi_{i=1}^{p^{b-1}(p-1)/d_b} (n_{bi}/n'_{bi})]\bar{n}/\overline{n'}$  distinct $G$-submodules of $P_l(V\setminus S_V)$ isomorphic to $H$. 
Now apply Corollary \ref{NSExt}.
\end{proof}

Let $\sigma$ be a generator of a cyclic group $G$ of order $p^a$, $c$ be a positive integer and  $M$ be a $\ZZ/l^c\ZZ[G]$-module. Let $\zeta_{p^b}$ be a primitive $p^b$-th root of unity in $\mathbb{C}$, $\Phi_{p^b}$ be the monic polynomial of $\zeta_{p^b}$ over $\QQ$ and $r_b=\frac{p^{b-1}(p-1)}{d_b}$ for $1\leq b\leq a$. Let the minimal primes of $l\ZZ[\zeta_{p^b}]$ be $Q_{b1},\ldots, Q_{br_b}$.
Then we have a $G$-stable decomposition 
\begin{align*}
(\ZZ/l^c\ZZ)[G] \cong \frac{\ZZ[x]}{(l^c,x^{p^a}-1)}
                   &= \frac{{\ZZ[x]}/{(x-1)} \oplus (\oplus_{b=1}^a ({\ZZ[x]}/(\Phi_{p^b}(x))))}{(\overline{l^c})}\\ 
                   &= \frac{(\ZZ/{l^c}\ZZ)[x]}{(x-1)} \oplus \frac{\ZZ[\zeta_p]}{(l^c)} \oplus\ldots\oplus\frac{\ZZ[\zeta_{p^a}]}{(l^c)}\\
                   &= \frac{(\ZZ/l^c\ZZ)[x]}{(x-1)} \oplus (\oplus_{b=1}^a (\oplus_{i=1}^{r_b} \frac{\ZZ[\zeta_{p^b}]}{Q_{bi}^c})).
\end{align*}

Hence $M\cong \ker(\sigma-\textbf{I})\oplus (\oplus_{b=1}^a (\oplus_{j=1}^{r_b} \frac{M}{Q_{bj}^cM}))$ (some of the summands may be trivial).
\begin{lm}
Let the hypothesis be as above.
Then $M$ can be expressed as a direct sum of indecomposable $G$-submodules isomorphic to $\ZZ/l^i\ZZ$, $\ZZ[\zeta_{p^b}]/Q_{bj}^i, 1\leq i\leq c$, $1\leq b\leq a$. The number of direct summands isomorphic to $\ZZ/l^i\ZZ$ is $f'_{i-1}-f'_i$ where $f'_i=\dim_{(\ZZ_l/l\ZZ_l)}(\frac{l^iM^G}{l^{i+1}M^G})$. The number of summands isomorphic to $\ZZ[\zeta_{p^b}]/Q_{bj}^i$ is $f_{b,i-1,j}-f_{bij}$ where $f_{bij}=\dim_{(\ZZ[\zeta_{p^b}]/Q_{bj})}(\frac{Q_{bj}^iN_{bj}}{Q_{bj}^{i+1}N_{bj}})$ and $N_{bj}=M/Q_{bj}^cM$. 
\end{lm}

\begin{proof}
Let $R=\mathbb{Z}[\zeta_{p^b}]$.
As $N_{bj}$ is a $R/{Q_{bj}^c}$-module and $R/Q_{bj}^c\cong R_{Q_{bj}}/Q_{bj}^cR_{Q_{bj}}$, $N_{bj}$ is a torsion module over $R_{Q_{bj}}$. Hence we have the following decomposition as $G$-modules:
$$N_{bj}\cong \oplus_{i=1}^c (R_{Q_{bj}}/(Q_{bj}^i))^{f_{b,i-1,j}-f_{bij}} \cong \oplus_{i=1}^c(R/Q_{bj}^i)^{f_{b,i-1,j}-f_{bij}}.$$
Note that $\ker(\sigma-\id)$ is a $\ZZ/l^c\ZZ$-module with trivial $G$-action. Then again, $M^G= \ker(\sigma-\id)\cong \oplus_{i=1}^c \oplus^{f'_{i-1}-f'_i}(\ZZ/l^i\ZZ)$. Then we get the $G$-stable decomposition into indecomposable (or zero) $G$-submodules: 
\begin{equation*}
M\cong (\oplus_{i=1}^c (\ZZ/l^i\ZZ)^{f'_{i-1}-f'_i})\oplus(\oplus_{b=1}^a\oplus_{j=1}^{r_b} \oplus_{i=1}^c (\ZZ[\zeta_{p^b}]/Q_{bj}^i)^{f_{b,i-1,j}-f_{bij}})
\end{equation*} 
\end{proof}

\begin{co}\label{5.6}
Let $H=\oplus_{i=0}^c(\ZZ/l^i\ZZ)^{e_i}$ for non-negative integers $e_i$ and $M= P_{l^c}(V\setminus S_V)$. Let $f_i'$ and $f_{bij}$ be as in the above lemma.
If $e_i$ can be expressed as $e_i=e'_i+\sum_{b=1}^a d_be''_{bi}$ for $0\leq e'_i\leq f'_{i-1}-f'_i$, $0\leq e''_{bi}\leq\sum_{j=1}^{r_b}(f_{b,i-1,j}-f_{bij})$, $1\leq i\leq c$, then the embedding problem  $(\beta:H\rtimes_\theta G\twoheadrightarrow G, \alpha:\pi_1^{\e}(X\setminus S_X)\twoheadrightarrow G)$ has a proper solution for some group homomorphism $\theta:G\longrightarrow \Aut(H)$.
\end{co}

\begin{proof}
Let $N_{bij}=R/Q_{bj}^i\cong R_{Q_{bj}}/(Q_{bj}R_{Q_{bj}})^i$. It is a module over $\ZZ/l^i\ZZ$. Note that $N_{bij}/lN_{bij}\cong R_{Q_{bj}}/Q_{bj}R_{Q_{bj}}$ and the latter is of dimension $d_b$ over $\ZZ/l\ZZ$. By Nakayama's lemma $N_{bij}$ has a minimal generating set of cardinality $d_b$ over $\ZZ/l^i\ZZ$. Since $|N_{bij}|=l^{id_b}$, we have an isomorphism of abelian groups $R/Q_{bj}^i\cong (\ZZ/l^i\ZZ)^{d_b}$. Then $(\ZZ/l^i\ZZ)^{e'_i}\oplus (\oplus_{b=1}^a(\ZZ/l^i\ZZ)^{d_b})^{e''_{bi}}$ can be identified as a $G$-stable submodule of $(\ZZ/l^i\ZZ)^{f'_{i-1}-f'_i}\oplus( \oplus_{b=1}^a\oplus_{j=1}^{r_b} (\ZZ[\zeta_{p^b}]/(t_{bj}^i))^{f_{b,i-1,j}-f_{bij}})$ (which is a direct summand of $P_{l^c}(V\setminus S_V)$ by the above lemma). Hence we get a $G$-module monomorphism from $H$ into $P_{l^c}(V\setminus S_V)$. The result now follows from Theorem \ref{solutionsEP-submodule}.
\end{proof}

\end{document}